\documentclass{amsart}

\usepackage[utf8]{inputenc}
\usepackage[T1]{fontenc}
\usepackage{lmodern}
\usepackage{amsmath,amssymb}
\usepackage{tikz-cd}
\usepackage{hyperref}

\title[Fano manifolds with complete intersection type VMRT]{Fano Manifolds with Complete Intersection Type VMRT}
\author{Cenhao Li}
\address{School of Mathematics, Shandong University, Jinan, Shandong 250100, China}
\email{202311781@mail.sdu.edu.cn}

\newtheorem{lemma}{Lemma}[section]
\newtheorem{theorem}{Theorem}[section]
\newtheorem{corollary}{Corollary}[section]
\newtheorem{definition}{Definition}[section]
\newtheorem{proposition}{Proposition}[section]

\newtheorem*{maintheorem}{Main Theorem}

\begin{document}
\begin{abstract}
We study Fano manifolds of Picard number one equipped with a locally flat VMRT structure whose VMRT at a general point is a smooth linearly non-degenerate complete intersection. We prove that such a manifold is biregular to the standard hyperquadric provided that one of two additional assumptions is satisfied. The proof uses equivariant compactifications of the vector group $\mathbb C^n$, Euler-symmetric projective varieties, fundamental forms, and special birational transformations.
\end{abstract}
\maketitle

\section{Introduction}\label{26-1-2-691}
We work over the field of complex numbers. For a uniruled projective manifold $X$, we denote by $\operatorname{RatCurves}^n(X)$ the normalized space of rational curves on $X$ (for the construction and properties of the scheme structure on $\operatorname{RatCurves}^n(X)$, we refer to Chapter II of \cite{1}). An irreducible component $\mathcal{K}$ of $\operatorname{RatCurves}^n(X)$ is called a family of minimal rational curves if, for a general point $x\in X$, the subscheme $\mathcal{K}_x$ of $\mathcal{K}$ consisting of rational curves passing through $x$ is non-empty and projective. Let
\[
\rho:\mathcal{U}\to \mathcal{K},\qquad \mu:\mathcal{U}\to X
\]
be the universal family and its evaluation morphism. The tangent map
\[
\tau:\mathcal{U}\dashrightarrow \mathbb{P}(TX)
\]
is defined at a point $u\in \mathcal{U}$ satisfying the following condition: the curve $\mu(\rho^{-1}(\rho(u)))$ is smooth at $\mu(u)$. In this case, $\tau$ sends $u$ to the tangent direction of this curve at $\mu(u)$. Let $\mathcal{C}\subset \mathbb{P}(TX)$ be the closure of the image of $\tau$; it is called the VMRT structure associated with $\mathcal{K}$. For a general point $x\in X$, the fiber $\mathcal{C}_x\subset \mathbb{P}(T_xX)$ is called the VMRT at $x$ associated with the family $\mathcal{K}$.

The core philosophy of the VMRT theory, developed primarily by Hwang and Mok, is that in many cases the local projective geometric properties of the VMRT structure on a uniruled projective manifold $X$ can control the global algebraic structure of $X$. A central problem in this area is the characterization problem of uniruled projective manifolds (see Problem 1.3 in \cite{2}): Given a (well-known) uniruled projective manifold $S$, what algebro-geometric conditions must the VMRT at a general point of an unknown uniruled projective manifold $X$ satisfy to guarantee that $X$ is biregular to $S$?

When we focus on Fano manifolds $X$ of Picard number 1 ($b_2(X)=1$), this problem exhibits a profound connection with Cartan geometry, particularly with cone structures. The smooth hyperquadric $\mathbb{Q}^n\subset \mathbb{P}^{n+1}$ is one of the most fundamental models in VMRT theory. It is well known that the VMRT at a general point $x$ of $\mathbb{Q}^n$ is projectively isomorphic to $\mathbb{Q}^{n-2}\subset \mathbb{P}^{n-1}$. On the other hand, Hwang proved the following:

\begin{theorem}[{\cite[Theorem~1.11]{3}}]\label{1233}
If a Fano manifold $X$ of Picard number 1 ($n=\dim(X)\ge 3$) has the property that its VMRT at a general point is a smooth hypersurface, then $X$ must be biregular to $\mathbb{Q}^{n}$.
\end{theorem}

Hwang's theorem solves the recognition problem for quadratic hypersurfaces in the case where the VMRT at a general point has codimension $1$. Beyond hypersurfaces, smooth complete intersections form one of the most natural classes of projective models. In fact, the phenomenon that the VMRT at a general point is a complete intersection is not rare: for example, when $X$ itself is a Fano complete intersection, its VMRT at a general point is again a complete intersection (see \cite[Example~1.4.2]{4}). This shows that a pointwise complete-intersection condition alone is not sufficient to impose effective structural restrictions. Therefore, this paper focuses instead on the existence of isotrivial VMRT structures of complete intersection type.

To make this precise, we recall two basic concepts. For a fixed projective variety $Z\subset \mathbb{P}^{n-1}$, if over a connected open subset of a uniruled projective manifold $X$ ($n=\dim(X)$), the VMRT $\mathcal{C}_x$ at a general point is projectively equivalent to $Z$, then the VMRT structure is said to be $Z$-isotrivial. Furthermore, if there exist an analytic open subset $U$ of $X$ and an open immersion $\varphi:U\to \mathbb{C}^n$ such that the restricted VMRT structure $\mathcal{C}|_U$ is equivalent to the trivial subbundle $\varphi(U)\times Z$ of $\mathbb{P} T(\varphi(U)) = \varphi(U) \times \mathbb{P}^{n-1}$, then the VMRT structure is said to be $Z$-locally flat.

Blowing up projective space along a linearly degenerate smooth closed subvariety $Z$ can produce a $Z$-locally flat VMRT structure (see \cite[Example~6.11]{7}). Thus, this paper turns to the existence of isotrivial VMRT structures of complete intersection type on Fano manifolds of Picard number $1$. Fu--Hwang proved that, for many smooth complete intersections $Z$, a $Z$-isotrivial VMRT structure is necessarily locally flat. More precisely, they proved the following result:

\begin{theorem}[{\cite[Theorem~1.7]{6}}]\label{26-1-2-660}
Let $Z\subset \mathbb{P}^{n-1}$ be a smooth linearly non-degenerate complete intersection.
Let us denote its multi-degree by $(m_1,\cdots,m_c)$, where $c$ is the codimension of $Z$. Assume further that $Z$ satisfies one of the following:

(i) $Z$ is a curve of multi-degree different from $(3)$, $(4)$, $(2,2)$, $(2,3)$ or $(2,2,2)$;

(ii) $Z\subset \mathbb{P}(V)$ is covered by lines, with multi-degree different from $(2)$, $(3)$, or $(2,2)$;

(iii) $Z$ is contained in a smooth hypersurface of degree $d \ge 3$ and its multi-degree $(m_1,\cdots,m_c)$ satisfies $d = m_1 < d + 2 \le m_2 \le\cdots\le m_c$.

Then any $Z$-isotrivial VMRT-structure on a uniruled projective manifold is locally flat.
\end{theorem}

Therefore, locally flat VMRT structures of complete intersection type form a natural case in the study of $Z$-isotrivial VMRT structures. The main result of this paper shows that, under some additional assumptions, such structures do not exist except for the VMRT structure of a hyperquadric.

\begin{maintheorem}\label{26-1-2-661}
Let $X$ be a Fano manifold of Picard number 1 with dimension $n\ge 3$. Let $D$ denote the ample generator of $\operatorname{Pic}(X)$, and let $m$ be the smallest positive integer such that $mD$ is very ample. Given the following data:

(i) $Z\subsetneq \mathbb{P}^{n-1}$ is a fixed linearly non-degenerate smooth complete intersection.

(ii) a family of minimal rational curves $\mathcal{K}$ on $X$, such that the associated VMRT structure $\mathcal{C}$ is $Z$-locally flat.

If one of the following conditions is satisfied:

(a) $H^0(\mathbb{P}^{n-1},\mathcal{I}_Z(m))=0$.

(b) $\dim(Z)\ge \dim (X)-4$.

Then $X$ is biregular to the standard hyperquadric $\mathbb{Q}^n$.
\end{maintheorem}

To prove the Main Theorem, our core idea is to construct a bridge extending from differential geometry to the classification theorem of quadratically symmetric varieties. According to a classical theorem by Fu and Hwang (\cite[Proposition~6.13]{7}), a Fano manifold $X$ of Picard number 1 with a locally flat VMRT structure, whose general VMRT is smooth and linearly non-degenerate, is necessarily an equivariant compactification of the complex vector group $\mathbb{C}^n$. Furthermore, according to the theory of Fu and Hwang, $X$ is actually an Euler symmetry variety under the embedding $X\subset \mathbb{P} H^0(X,mD)^*$ (\cite[Theorem~5.5]{8}). 

Since the geometric structure of an Euler symmetry variety is completely governed by its fundamental forms at a general point, we incorporate the assumption that ``the VMRT at a general point is a complete intersection'' into the theoretical framework of Euler symmetry varieties. Using the classical theory of fundamental forms founded by Cartan (\cite[Theorem~3.3]{8}), we will prove that the algebraic restrictions of the complete intersection will force the higher-order ranks of the symbol systems to vanish entirely, thereby deducing that $X$ is a quadratically symmetric variety. Since quadratically symmetric prime Fano manifolds have been completely classified (\cite[Theorem~7.8]{9}), we determine that $X$ is isomorphic to the hyperquadric. 

Condition (a) will be used to determine that the rational map $\mathbb{P}^n\dashrightarrow \mathbb{P} H^0(X,mD)^*$ given by \[ \mathbb{C}^n\subset X\subset \mathbb{P} H^0(X,mD)^* \] is actually a special birational transformation of type $(m+1,1)$ (see Definition~\ref{26-1-2-680}); then Proposition~\ref{prop:linearly-degenerate-base-locus} reduces it to type $(2,1)$, and the Fu--Hwang classification of such transformations forces the VMRT to be a hyperquadric.

Condition (b) forces the Fano index into the range \(i(X)\ge n-2\). Hence one may apply the Fu--Montero classification theorem for equivariant compactifications of vector groups with high index \cite[Theorem~1.1]{14} to exclude all candidate models other than the hyperquadric.

\subsection*{Standing assumptions}
In the next three sections, unless otherwise specified, we fix $X$ to be an $n$-dimensional Fano manifold of Picard number 1, and $\mathcal{K}$ to be a family of minimal rational curves on it. We assume that the associated VMRT structure is locally flat, and for a general point $x\in X,\ \mathcal{C}_x\subset \mathbb{P}(T_xX)$ is a non-singular irreducible linearly non-degenerate projective subvariety strictly contained in $\mathbb{P}(T_xX)$.

\section{The \texorpdfstring{$\mathbb{C}^n\rtimes \mathbb{C}^\times$}{Cn rtimes Cstar}-action on \texorpdfstring{$X$}{X}}\label{26-1-2-689}
As mentioned in the introduction, the local flatness of the VMRT structure $\mathcal{C}$ means that in an analytic neighborhood of a general point, $\mathcal{C}$ is equivalent to a trivial cone structure on an open subset $U\subset \mathbb{C}^n$ with coordinates $z_1,\ldots,z_n$. This local flatness guarantees that the holomorphic vector fields $\frac{\partial}{\partial z_1},\cdots,\frac{\partial}{\partial z_n}$, together with the Euler vector field
\[
E=z_1\frac{\partial}{\partial z_1}+\cdots+z_n\frac{\partial}{\partial z_n},
\]
preserve the VMRT structure $\mathcal{C}$ on $U$. According to the Cartan-Fubini type extension theorem developed by Hwang and Mok (\cite[Main Theorem]{11}), these holomorphic vector fields, which locally preserve the VMRT structure, can be uniquely analytically continued to the entire manifold $X$.

\begin{proposition}\label{26-1-2-662}
There exists an algebraic $\mathbb{C}^n\rtimes \mathbb{C}^\times$-action on $X$ satisfying:

1. The restriction of this action to $\mathbb{C}^n$ endows $X$ with the structure of an equivariant compactification of $\mathbb{C}^n$. That is, this $\mathbb{C}^n$-action has an open orbit $X^o$ (since $\mathbb{C}^n$ has no non-trivial finite subgroups, this open orbit is isomorphic to $\mathbb{C}^n$).

2. Under the aforementioned isomorphism $X^o\cong \mathbb{C}^n$, there is a commutative diagram:
\[\begin{tikzcd}
	{(\mathbb{C}^n\rtimes \mathbb{C}^\times)\times X} && X \\
	{(\mathbb{C}^n\rtimes \mathbb{C}^\times)\times \mathbb{C}^n} && {\mathbb{C}^n}
	\arrow[from=1-1, to=1-3]
	\arrow[hook, from=2-1, to=1-1]
	\arrow["{((v,\alpha),x)\mapsto v+\alpha\cdot x}"', from=2-1, to=2-3]
	\arrow[hook, from=2-3, to=1-3]
\end{tikzcd}\]
\end{proposition}
\begin{proof}
The equivariant compactification structure follows from \cite[Proposition~6.13]{7}, while the extension of the scalar $\mathbb{C}^\times$-action follows from \cite[Proposition~5.4(iii)]{8}. Combining these two facts gives the $(\mathbb{C}^n\rtimes \mathbb{C}^\times)$-action.
\end{proof}

Regarding the origin $0\in \mathbb{C}^n\cong X^o$ as a distinguished point $x\in X$, we emphasize that the restriction of the aforementioned action to $\mathbb{C}^\times$ is induced by the Euler vector field, and thus it preserves the projective tangent space $\mathbb{P}(T_xX)$ at $x$. We note that this $\mathbb{C}^\times$-action provides a one-parameter family of deformations for any minimal rational curve $l$ in $\mathcal{K}_x$. For a general point $x\in X$, Kebekus' theorem implies that the tangent map
from the normalization of $\mathcal{K}_x$ to $\mathcal{C}_x$ is finite (see \cite[Theorem~3.4]{19}). Hence this one-parameter deformation family is trivial, and therefore $l$ is the closure in $X$ of a one-dimensional linear subspace of $\mathbb{C}^n\cong X^o$. Translating the point $x$ by the action of $\mathbb{C}^n$ on $X$, we obtain

\begin{proposition}\label{26-1-2-663}
For any $x\in X^o$ and any $l\in \mathcal{K}_x$, there exists a $\mathbb{C}^\times$-action on $X$ such that $l$ is the closure of some orbit of this action.
\end{proposition}

The argument above is the application of the method of \cite[Proposition~5.4(iv)]{8} to the present setting.

We can actually compute the intersection number of a minimal rational curve $l \in \mathcal{K}$ with the generator of $\operatorname{Pic}(X)$. Take an arbitrary $(n-1)$-dimensional subspace $\mathbb{C}^{n-1} \subset \mathbb{C}^n \subset X$, and let $H$ be its closure in $X$. Since $\mathcal{C}_x$ is non-degenerate, for a general $l\in \mathcal{K}$, the intersection $l\cap X^o$ and $H\cap X^o=\mathbb{C}^{n-1}$ intersect transversally at exactly one point. Let $B:=X\backslash X^o$ be the boundary divisor. Since $B\cap H$ is a subvariety of $X$ of codimension 2, by the deformation theory of minimal rational curves (see, for example, \cite[Lemma 1.1(3)]{11}), for a general minimal rational curve $l\in \mathcal{K}$, the intersection $l\cap B\cap H$ is empty. This gives the following conclusion; its proof is the same as the corresponding argument in \cite[Proposition~5.4(v)]{8}.

\begin{proposition}\label{26-1-2-664}
For any $l\in \mathcal{K},\ l\cdot H=1$, and $H$ is the generator of $\operatorname{Pic}(X)$.
\end{proposition}

We say that a Fano manifold $\widetilde{X}$ of Picard number 1 is a prime Fano manifold if the ample generator $D$ of $\operatorname{Pic}(\widetilde{X})$ is very ample, and $\widetilde{X}$ is covered by lines under the embedding $\widetilde{X}\subset \mathbb{P} H^0(\widetilde{X},D)^*$. The proposition above shows that for the given $X$, when the ample generator of $\operatorname{Pic}(X)$ is very ample, $X$ is a prime Fano manifold.

\begin{lemma}[{\cite[Theorem~3.1]{18}}]\label{323}
Let $Z\subset \mathbb{P}^n$ be a linearly non-degenerate smooth complete intersection of positive codimension. If $Z$ is not projectively isomorphic to the hyperquadric and $\dim(Z)\ge 3$, then the automorphism group $\operatorname{Aut}(Z)$ is finite.
\end{lemma}

\begin{lemma}[{\cite[Lemma~5.6]{3}}]\label{324}
Let $Z\subset \mathbb{P}^{n-1}$ be a linearly non-degenerate non-singular projective subvariety with $\mathfrak{aut}(Z)=0$, where $\mathfrak{aut}(Z)$ denotes the Lie algebra of the automorphism group $\operatorname{Aut}(Z)$. For the flat $Z$-isotrivial cone structure $\mathbb{C}^n\times Z\subset \mathbb{P} T\mathbb{C}^n$ on $\mathbb{C}^n$, the Lie algebra of germs of holomorphic vector fields at $0\in \mathbb{C}^n$ preserving the cone structure is $\mathrm{Lie}(\mathbb{C}^n\rtimes \mathbb{C}^\times)=\mathbb{C}^n\rtimes \mathbb{C}$.
\end{lemma}

\begin{proposition}\label{310}
Suppose that $X$ is not isomorphic to the hyperquadric. If for a general point $x\in X$, the VMRT $\mathcal{C}_x$ is a linearly non-degenerate smooth complete intersection and $\dim(\mathcal{C}_x)\ge 3$, then the Lie algebra $\mathfrak{aut}(X)$ is isomorphic to $\mathbb{C}^n\rtimes \mathbb{C}$.
\end{proposition}

\begin{proof}
By Theorem \ref{1233}, $\mathcal{C}_x$ is not projectively isomorphic to $\mathbb{Q}^{n-2}$. Consequently, by Lemma \ref{323}, the automorphism group $\operatorname{Aut}(\mathcal{C}_x)$ is finite. Thus, it follows from Lemma \ref{324} that the Lie algebra of local holomorphic vector fields preserving the VMRT structure is $\mathbb{C}^n\rtimes(\mathbb{C}\cdot E)$, where $E$ denotes the Euler vector field. This implies that the Lie algebra homomorphism $\mathbb{C}^n\rtimes \mathbb{C} \to \mathfrak{aut}(X)$ induced by the $\mathbb{C}^n\rtimes \mathbb{C}^\times$-action on $X$ is an isomorphism.
\end{proof}

Moreover, Fu--Hwang used the Cartan--Fubini type extension theorem \cite[Main Theorem]{11} to prove the following uniqueness result for equivariant compactifications:

\begin{theorem}[{\cite[Theorem~1.2]{12}}]\label{26-1-2-665}
Up to isomorphism, the equivariant compactification structure on $X$ is unique.
\end{theorem}

Note that $\mathbb{P}^n$ admits many non-equivalent structures as an equivariant compactification of $\mathbb{C}^n$ (see \cite{13}); however, our assumption that $\mathcal{C}_x$ is strictly contained in $\mathbb{P}(T_xX)$ at a general point $x$ excludes the case $X\cong \mathbb{P}^n$.

\section{Fundamental forms on \texorpdfstring{$X$}{X}}\label{26-1-2-688}
To manifest the intrinsic equivariant compactification structure on $X$ described in Section \ref{26-1-2-689}, we consider the projective embedding $X\subset \mathbb{P} H^0(X,mD)^*$ of $X$, where $D$ is the ample generator of $\operatorname{Pic}(X)$, and $m$ is the smallest positive integer such that $mD$ is very ample.

First, we introduce the concept of Euler symmetry varieties proposed by Fu and Hwang \cite{8}:
\begin{definition}\label{26-1-2-666}
Let $Y\subset \mathbb{P}(V)$ be a linearly non-degenerate projective variety. For a nonsingular point $y\in \operatorname{Reg}(Y)$, a $\mathbb{C}^\times$-action on $Y$ is said to be of Euler type at $y$ if

(i) The $\mathbb{C}^\times$-action can be extended to a $\mathbb{C}^\times$-action on $\mathbb{P}(V)$.

(ii) $y$ is an isolated fixed point of this $\mathbb{C}^\times$-action on $Y$.

(iii) The induced action on $T_yY$ is given by scalar multiplication.

We say $Y$ is an Euler symmetry variety, if for a general point $y$ on $Y$, there exists a $\mathbb{C}^\times$-action on $Y$ which is of Euler type at $y$.
\end{definition}

By Proposition~\ref{26-1-2-662}, for every point $x\in X^o\cong \mathbb{C}^n$, there exists a $\mathbb{C}^\times$-action with $x$ as an isolated fixed point such that its induced action on $T_xX$ is scalar multiplication. Since $\rho(X)=1$, this action preserves the ample generator $D$ of $\operatorname{Pic}(X)$, and hence induces a linear $\mathbb{C}^\times$-action on $\mathbb{P}H^0(X,mD)^*$. This gives the following conclusion; its proof is the same as the corresponding argument in \cite[Theorem~5.5]{8}.

\begin{proposition}\label{26-1-2-667}
$X\subset \mathbb{P} H^0(X,mD)^*$ is an Euler symmetry variety.
\end{proposition}

Euler symmetry varieties exhibit extremely strong global rigidity in projective algebraic geometry. To describe this property, we need to introduce the concept of the fundamental forms of a projective variety $Y\subset \mathbb{P}(V)$ at a point:

\begin{definition}\label{26-1-2-668}
Let $y \in Y \subset \mathbb{P}(V)$ be a nonsingular point of a linearly non-degenerate projective variety. Let $L$ be the line bundle on $Y$ induced by the hyperplane line bundle on $\mathbb{P}(V)$. Let $\mathbf{m}_y\subset \mathcal{O}_{Y,y}$ be the maximal ideal of the local ring of $Y$ at $y$, and for each non-negative integer $k$, let $\mathbf{m}_y^k$ denote its $k$-th power. For a section $s \in H^0(Y,L)$, let $j^k_y(s)$ be the $k$-jet of $s$ at $y$ with $j^0_y(s) = s_y \in L_y$. We have a descending filtration of the dual space $V^* \subset H^0(Y,L)$:

$$V^* \cap \ker(j^k_y) \subset V^* \cap \ker(j^{k-1}_y)$$
This induces an injective homomorphism:

$$V^* \cap \ker(j^{k-1}_y)/V^* \cap \ker(j^k_y) \to L_y \otimes \mathrm{Sym}^k T_y^*Y$$
For each $k \geq 2$, the subspace $F^k_y \subset \mathrm{Sym}^k T_y^*Y$ defined by the image of this homomorphism is called the $k$-th fundamental form of $Y$ at $y$. For convenience, set $F^0_y = \mathrm{Sym}^0 T_y^*Y = \mathbb{C}$ and $F^1_y = \mathrm{Sym}^1 T_y^*Y = T_y^*Y$. The collection of subspaces $\mathbf{F}_y := \bigoplus_{k \geq 0} F^k_y \subset \bigoplus_{k \geq 0} \mathrm{Sym}^k T_y^*Y$ is called the system of fundamental forms of $Y$ at $y$. 
\end{definition}

For the explicit expressions of the above definition of the system of fundamental forms of $Y$ at $y$ in local holomorphic coordinate systems, one can refer to \cite[Lemma 2.5]{8}. Fu and Hwang proved in \cite[Proposition~2.7]{8} that:

\begin{theorem}[{\cite[Proposition~2.7]{8}}]\label{26-1-2-669}
An Euler symmetry variety $Y\subset \mathbb{P}(V)$ is determined, up to projective equivalence, by its system of fundamental forms at a general point $y$.
\end{theorem}

The theorem above shows that the system of fundamental forms at a general point is sufficient to characterize the corresponding Euler symmetry variety. To apply this structure to the equivariant compactification structure on $X$, we first introduce the following definition:

\begin{definition}\label{26-1-2-670}
Let $y \in Y \subset \mathbb{P}(V)$ be a nonsingular point of a non-degenerate projective variety. Let $r \geq 1$ satisfy $F^r_y \neq 0$ and $F^{r+i}_y = 0$ for all $i \geq 1$. Then $r$ is called the rank of $\mathbf{F}_y$. Define an embedding:

$$\phi_{\mathbf{F}}: T_yY \hookrightarrow \mathbb{P}(\mathbb{C} \oplus T_yY \oplus (F^2_y)^* \oplus \cdots \oplus (F^r_y)^*)$$
$$w \mapsto [1: w: (\phi_2 \mapsto \phi_2(w,w)): \cdots: (\phi_r \mapsto \phi_r(w,\cdots,w))]$$
Write $V_{\mathbf{F}}:=\mathbb{C} \oplus T_yY \oplus (F^2_y)^* \oplus \cdots \oplus (F^r_y)^*$. Let $M(\mathbf{F}_y) := \overline{\phi_{\mathbf{F}}(T_yY)}\subset \mathbb{P}(V_{\mathbf{F}})$. 
\end{definition}

Fu and Hwang proved that:

\begin{theorem}[{\cite[Theorem~3.7]{8}}]\label{26-1-2-671}
Let $y \in Y \subset \mathbb{P}(V)$ be a nonsingular point of a non-degenerate projective variety. Then

1. The translation action of $T_yY$ on itself as a complex vector group can be extended to an action of $T_yY$ on $M(\mathbf{F}_y)$, and $\phi_{\mathbf{F}}(T_yY)$ is the open orbit of this action.

2. The system of fundamental forms of $M(\mathbf{F}_y)\subset \mathbb{P}(V_{\mathbf{F}})$ at a general point is isomorphic to $\mathbf{F}_y$.
\end{theorem}

We now apply the above theorem to $X$. Take an arbitrary point $x\in X^o$. Theorem \ref{26-1-2-671} and Theorem \ref{26-1-2-669} show that $X\subset \mathbb{P} H^0(X,mD)^*$ is projectively isomorphic to $M(\mathbf{F}_x)\subset \mathbb{P}(V_{\mathbf{F}})$, where $x$ corresponds to the image of $0\in T_xX$ under the map $\phi_{\mathbf{F}}:T_xX\hookrightarrow M(\mathbf{F}_x)=X$. Therefore, by Theorem \ref{26-1-2-671}, we can express the $\mathbb{C}^n\rtimes \mathbb{C}^\times$-action on $X$ through the system of fundamental forms at $x$, specifically:

1. The open orbit is precisely $\phi_{\mathbf{F}}(T_xX)$.

2. The $\mathbb{C}^\times$-action on $X$ is exactly the restriction of the weighted scalar multiplication in $\mathbb{P}(V_{\mathbf{F}})$ given by
\[\lambda\cdot[t:w:f_2:\cdots:f_r]=[t:\lambda w: \lambda^2 f_2:\cdots:\lambda^r f_r]\]
to $X=M(\mathbf{F}_x)$.

Next, we further investigate how the minimal rational curves in $\mathcal{K}$ can be expressed via the system of fundamental forms at a general point $x$. For this, we need to examine the base locus of polynomials in $F^k_x\subset \mathrm{Sym}^k T_x^* X$. For a general projective variety, we give the following definition:

\begin{definition}\label{26-1-2-672}
Let $y \in Y \subset \mathbb{P}(V)$ be a nonsingular point of a non-degenerate projective variety. For any $k \geq 1$, define the base locus of $F^k_y$ as:
$$\mathbf{Bs}(F^k_y) := \{[w] \in \mathbb{P}(T_yY) \mid \phi(w,\cdots,w) = 0 \text{ for all } \phi \in F^k_y\}$$
A positive integer $l$ is called the order of the system of fundamental forms $\mathbf{F}_y$ if $\mathbf{Bs}(F^l_y)=\emptyset$ and $\mathbf{Bs}(F^{l+1}_y)\ne \emptyset$.
\end{definition}

The following classical theorem of Cartan gives the prolongation property of fundamental forms at a general point:
\begin{theorem}[{\cite[Theorem~3.3]{8}}]\label{26-1-2-673}
Let $y \in Y \subset \mathbb{P}(V)$ be a general point of a non-degenerate projective variety. For any $k \geq 1$, $w \in T_yY$, and $\phi \in F^{k+1}_y$, the map $\mathrm{Sym}^k T_yY \to \mathbb{C}$ defined by $w_1 \odot \cdots \odot w_k \mapsto \phi(w,w_1,\cdots,w_k)$ is an element of $F^k_y$.
\end{theorem}

As an immediate corollary of this theorem, we obtain
\begin{corollary}\label{26-1-2-674}
$\emptyset=\mathbf{Bs}(F^1_y)\subset  \mathbf{Bs}(F^2_y)\subset \cdots\subset \mathbf{Bs}(F^r_y)\subset \mathbf{Bs}(F^{r+1}_y)=\mathbb{P}(T_yY)$.
\end{corollary}

Returning to the analysis on $X$. For any $l\in \mathcal{K}_{x}$, by the proof of Proposition \ref{26-1-2-663}, there exists $[w]\in \mathbb{P}(T_xX)$ such that $l$ is equal to the rational normal curve
\[\overline{\phi_{\mathbf{F}}(\mathbb{C}\cdot w)}\subset \mathbb{P} H^0(X,mD)^*\]
Since $l\cdot D=1$ (Proposition \ref{26-1-2-664}), the degree of $l$ in $\mathbb{P} H^0(X,mD)^*$ is $m$. In particular, the tangent direction $[w]$ of $l$ at $x$ lies in $\mathbf{Bs}(F^{m+1}_x)$. On the other hand, since $D$ is ample, any rational curve on $X$ has degree at least $m$ in $\mathbb{P} H^0(X,mD)^*$. In particular, for any $[w]\in \mathbb{P}(T_xX)$, the degree of the rational normal curve $\overline{\phi_{\mathbf{F}}(\mathbb{C}\cdot w)}\subset \mathbb{P} H^0(X,mD)^*$ is at least $m$. Therefore, by Corollary \ref{26-1-2-674}, we have $\mathbf{Bs}(F^{m}_x)=\emptyset$. This gives the concrete form, in the present setting, of the method in \cite[Proposition~4.3]{8} and \cite[Theorem~5.5]{8}. Thus we obtain the following conclusion.
\begin{proposition}\label{26-1-2-675}
For a general $x\in X,\ \mathcal{C}_x\subset \mathbf{Bs}(F^{m+1}_x)$, and the order of $\mathbf{F}_x$ is equal to $m$.
\end{proposition}

The inclusion $\mathcal{C}_x\subset \mathbf{Bs}(F^{m+1}_x)$ mentioned above is not necessarily an equality, because $X$ may not possess only a single family of minimal rational curves. However, for any $[w]\in \mathbf{Bs}(F^{m+1}_x)$, the deformation family of the rational normal curve $\overline{\phi_{\mathbf{F}}(\mathbb{C}\cdot w)}$ given by the $\mathbb{C}^n$-action on $X$ dominates $X$; hence, there exists a family of minimal rational curves $\mathcal{K}'$ on $X$ such that its VMRT $\mathcal{C}'_x$ at $x$ contains $[w]$. Thus, we have:

\begin{proposition}\label{26-1-2-676}
1. $\mathcal{C}_x$ is an irreducible component of $\mathbf{Bs}(F^{m+1}_x)$. In particular, since $\mathbf{Bs}(F^{m+1}_x)$ is the intersection of $\dim(F^{m+1}_x)$ hypersurfaces, we have $\dim(\mathcal{C}_x)\ge n-1-\dim(F^{m+1}_x)$.

2. If $\mathcal{K}$ is the unique family of minimal rational curves on $X$, then $\mathcal{C}_x=\mathbf{Bs}(F^{m+1}_x)$.
\end{proposition}
\begin{proof}
1. By assumption, $\mathcal{C}_x$ is irreducible. Let $B_0$ be an irreducible component of $\mathbf{Bs}(F^{m+1}_x)$ containing $\mathcal{C}_x$. By the preceding discussion, for any point $[w]$ of $B_0$, there exists a family of minimal rational curves $\mathcal{K}'$ such that its VMRT $\mathcal{C}'_x$ at $x$ contains $[w]$, and every member $l'$ of this family satisfies $D\cdot l'=1$. Let $i(X)$ denote the Fano index of $X$, namely the positive integer satisfying $-K_X=i(X)D$. Then
\[
\dim \mathcal{C}'_x=\dim \mathcal{K}'_x=-K_X\cdot l'-2=i(X)-2.
\]
By Proposition~\ref{26-1-2-664}, any member of the chosen family $\mathcal{K}$ also satisfies $D\cdot l=1$. Hence
\[
\dim \mathcal{C}_x=i(X)-2.
\]
Since there are only finitely many families of minimal rational curves of $D$-degree $1$, we have $\dim B_0=\dim \mathcal{C}_x$, and hence $B_0=\mathcal{C}_x$.

2. This follows from the preceding discussion.
\end{proof}

Let $I(\mathcal{C}_x)\subset \bigoplus_{k\ge 0}\mathrm{Sym}^kT^*_xX$ denote the ideal generated by all homogeneous polynomials vanishing on $\mathcal{C}_x$. The following proposition gives the sufficient criterion needed in this paper:

\begin{proposition}\label{26-1-2-677}
For a general $x\in X$, assuming that $\mathcal{C}_x\subset \mathbb{P}(T_xX)$ is a smooth linearly non-degenerate complete intersection satisfying $I(\mathcal{C}_x)_{m}=0$, then $I(\mathcal{C}_x)$ is the ideal generated by $F^{m+1}_x$.
\end{proposition}
\begin{proof}
Let $c$ be the codimension of $\mathcal{C}_x\subset \mathbb{P}(T_xX)$. Since $\mathcal{C}_x$ is a complete intersection, there exist homogeneous polynomials $f_1,\cdots,f_c$ that generate $I(\mathcal{C}_x)$. Suppose there are $k$ polynomials among $f_1,\cdots,f_c$ with degree equal to $m+1$. Since $I(\mathcal{C}_x)_{m}=0$, the vector space $I(\mathcal{C}_x)_{m+1}$ is linearly spanned by these $k$ polynomials, so $\dim(I(\mathcal{C}_x)_{m+1})=k$. On the other hand, the inclusions $F^{m+1}_x\subset I(\mathbf{Bs}(F^{m+1}_x))_{m+1}\subset I(\mathcal{C}_x)_{m+1}$ imply that $\dim(F^{m+1}_x)\le k$. By Proposition \ref{26-1-2-676}(1), we have $\dim(F^{m+1}_x)\ge c\ge k$. Therefore, $k=c$, and $I(\mathcal{C}_x)_{m+1}=F^{m+1}_x$.
\end{proof}

The preceding proposition further implies the following:
\begin{proposition}\label{26-1-2-678}
For a general $x\in X$, assuming that $\mathcal{C}_x\subset \mathbb{P}(T_xX)$ is a smooth linearly non-degenerate complete intersection and $I(\mathcal{C}_x)$ is the ideal generated by $F^{m+1}_x$, then the rank of the system of fundamental forms $\mathbf{F}_x$ is $r=m+1$.
\end{proposition}

Before proving this proposition, we use the following classical fact concerning complete intersections.
\begin{lemma}\label{26-1-2-679}
Let $Z\subset \mathbb{P}^n$ be a complete intersection of dimension $n-c$, and let $f\in\mathbb{C}[x_0,\cdots,x_n]$ be a homogeneous polynomial. If for any $i\in \{0,\cdots,n\}$, $\frac{\partial f}{\partial x_i}\in I(Z)$, then $f\in I(Z)^2$.
\end{lemma}
\begin{proof}
Let $g_1,\cdots,g_c$ be homogeneous polynomials generating $I(Z)$. By Euler's formula for homogeneous polynomials, we have
\[\deg(f)\cdot f=\sum_{k=0}^nx_k\frac{\partial f}{\partial x_k}\]
Thus $f\in I(Z)$. We can express $f$ as
\[f=a_1g_1+\cdots+a_c g_c\]
where $a_1,\cdots,a_c$ are homogeneous polynomials. Taking the partial derivative of both sides with respect to $x_k$ yields
\[\frac{\partial f}{\partial x_k}=\sum_{j=1}^c \frac{\partial a_j}{\partial x_k}g_j+\sum_{j=1}^c a_j\frac{\partial g_j}{\partial x_k}\]
Since $\frac{\partial f}{\partial x_k},g_1,\cdots,g_c\in I(Z)$, we have
\[\sum_{j=1}^c a_j\frac{\partial g_j}{\partial x_k}\in I(Z)\]

Let $w\in \mathbb{C}^{n+1}\backslash \{0\}$ be a preimage of a nonsingular point of $Z$ on the affine cone. Evaluating the above equation at $w$, we obtain
\[ 
\begin{pmatrix} a_1(w) & \cdots & a_c(w) \end{pmatrix} 
\begin{pmatrix} 
\frac{\partial g_1}{\partial x_0}(w) & \cdots & \frac{\partial g_1}{\partial x_n}(w) \\
\vdots & \ddots & \vdots \\
\frac{\partial g_c}{\partial x_0}(w) & \cdots & \frac{\partial g_c}{\partial x_n}(w) 
\end{pmatrix}
= \begin{pmatrix} 0 & \cdots & 0 \end{pmatrix} 
\]
which can be concisely written as the product of a row vector and the Jacobian matrix: $\mathbf{a}(w) \cdot \operatorname{Jac}_g(w) = \mathbf{0}$. Since $[w]\in Z$ is a nonsingular point of $Z$, the rank of $\operatorname{Jac}_g(w)$ is equal to the codimension $c$ of $Z$. Thus, $\mathbf{a}(w)=(a_1(w),\cdots,a_c(w))=\mathbf{0}$. Because $[w]$ is an arbitrary nonsingular point of $Z$, it follows that $a_1,\cdots,a_c\in I(Z)$, and therefore $f=\sum_{j=1}^ca_jg_j\in I(Z)^2$.
\end{proof}

\begin{proof}[Proof of Proposition \ref{26-1-2-678}]
By Corollary \ref{26-1-2-674}, it suffices to prove that $F^{m+2}_x=0$. Take a basis $\{e_1,\cdots,e_n\}$ for the tangent space $W=T_xX$, and let $\{z_1,\cdots,z_n\}\subset W^*$ be its dual basis. For any $f\in F^{m+2}_x\subset \mathrm{Sym}^{m+2}W^*$, as an element of $\mathrm{Sym}^{m+2}W^*$, $f$ corresponds to a homogeneous polynomial on $W$ (denoted by $f(z)$), as well as a symmetric $(m+2)$-multilinear form on $W$, which we denote by $\widetilde{f}:W\times \cdots \times W\to \mathbb{C}$. Their relationship is given by:
\[f(z)=\widetilde{f}(z,\cdots,z),\ \forall z\in W\]
Therefore,
\[\frac{\partial f}{\partial z_i}(z)=\underset{t\to 0}{\lim}\frac{\widetilde{f}(z+te_i,\cdots,z+te_i)-\widetilde{f}(z,\cdots,z)}{t}=(m+2)\widetilde{f}(e_i,z,\cdots,z)\]

By the prolongation property in Theorem~\ref{26-1-2-673}, the map $w_1\odot\cdots\odot w_{m+1}\to \widetilde{f}(e_i,w_1,\cdots,w_{m+1})$ defines an element of $F^{m+1}_x$. Thus,
\[\frac{\partial f}{\partial z_i}\in F^{m+1}_x,\ \forall i\in \{1,\cdots,n\}\]
From the condition, we know $I(\mathcal{C}_x)_{m+1}=F^{m+1}_x$. Thus, for any $i$, $\frac{\partial f}{\partial z_i}\in I(\mathcal{C}_x)$. Consequently, by Lemma \ref{26-1-2-679}, we obtain $f\in I(\mathcal{C}_x)^2$. Since $I(\mathcal{C}_x)$ is generated by a family of polynomials of degree $m+1$, we have $(I(\mathcal{C}_x)^2)_{m+2}=0$, which implies $f=0$.
\end{proof}

\section{Special Birational Transformations}\label{26-1-2-687}
In Section \ref{26-1-2-688}, we obtained the embedding
\[\phi_{\mathbf{F}}:\ T_xX\hookrightarrow \mathbb{P} H^0(X,mD)^*\]
whose image is the open orbit of $X$ under the action of the $n$-dimensional vector group. We can rewrite this embedding in the form of a birational transformation:
\[\Phi_{\mathbf{F}}:\ \mathbb{P}(\mathbb{C}\oplus T_xX)\dashrightarrow X\subset \mathbb{P} H^0(X,mD)^*\]
\[[1:w]\mapsto \phi_{\mathbf{F}}(w)\]
where the rational map $\Phi_{\mathbf{F}}$ is defined by a linear system in $\mathcal{O}_{\mathbb{P}(\mathbb{C}\oplus T_xX)}(r)$ (here $r$ denotes the rank of the system of fundamental forms at $x$), and its inverse $\Phi_{\mathbf{F}}^{-1}$ is defined by a linear system in $\mathcal{O}_{\mathbb{P} H^0(X,mD)^*}(1)$.

\begin{definition}\label{26-1-2-680}
A birational transformation
\[\Phi:\mathbb{P}^n\dashrightarrow Y\subset \mathbb{P}^N\]
(where $Y\subset \mathbb{P}^N$ is a nonsingular projective variety of Picard number 1) is called a special birational transformation of type $(a,b)$ if

1. The base locus $S\subset \mathbb{P}^n$ of $\Phi$ is irreducible and nonsingular.

2. The rational map $\Phi$ is given by a linear system of $\mathcal{O}_{\mathbb{P}^n}(a)$.

3. The rational map $\Phi^{-1}$ is given by a linear system of $\mathcal{O}_{Y}(b)$.
\end{definition}

\begin{proposition}\label{1321}
Let $Y\subset \mathbb{P}(V)$ be a nonsingular linearly non-degenerate Euler-symmetric variety of Picard number 1. Suppose that for a general point $y \in Y$, $\mathbf{F}_y$ has rank $r$ ($r \ge 2$) and order $r-1$. Assume that the scheme-theoretic base locus of $F^{r}_y$ on $\mathbb P(T_yY)$ is irreducible and nonsingular. Then the birational transformation
\[ \Phi_{\mathbf{F}}: \mathbb{P}(\mathbb{C}\oplus T_yY) \dashrightarrow Y \subset \mathbb{P}(V) \]
\[ [1:w] \mapsto \phi_{\mathbf{F}}(w) \]
is a special birational transformation of type $(r,1)$, and the base locus of $\Phi_{\mathbf{F}}$ is $\mathbf{Bs}(F^{r}_y)\subset \mathbb{P}(T_yY)$, which is linearly degenerate.
\end{proposition}
\begin{proof}
Since $\mathbf{F}_y$ has order $r-1$, it follows that $\mathbf{Bs}(F^{r-1}_y)=\emptyset$. Therefore, the base locus of the birational transformation $\Phi_{\mathbf{F}}$ is projectively equivalent to the scheme-theoretic base locus of $F^{r}_y$. 
\end{proof}

\begin{proposition}\label{26-1-2-681}
Assume that $\mathcal{C}_x$ is a smooth linearly non-degenerate complete intersection, and $I(\mathcal{C}_x)$ is the ideal generated by $F^{m+1}_x$. Then $\Phi_{\mathbf{F}}$ is a special birational transformation of type $(m+1,1)$, and its base locus is $\mathcal{C}_x\subset \mathbb{P}(\mathbb{C}\oplus T_xX)$.
\end{proposition}

\begin{proof}
By Proposition \ref{26-1-2-675}, the order of $\mathbf{F}_x$ is $m$, and by Proposition \ref{26-1-2-678}, the rank of $\mathbf{F}_x$ is $m+1$. Therefore, it follows from Proposition \ref{1321} that $\Phi_{\mathbf{F}}$ is a special birational transformation of type $(m+1,1)$. Since $I(\mathcal{C}_x)$ is the ideal generated by $F^{m+1}_x$, the scheme-theoretic base locus of $F^{m+1}_x$ is exactly $\mathcal{C}_x$.
\end{proof}

In the preceding proposition, we note that the base locus $\mathcal{C}_x$ of $\Phi_{\mathbf{F}}$ is contained in the hyperplane $\mathbb{P}(T_xX)$. The following proposition shows that a special birational transformation of arbitrary type must be of type $(2,1)$ as soon as its base locus is linearly degenerate.

\begin{proposition}\label{prop:linearly-degenerate-base-locus}
Let
\[
\Phi:\mathbb{P}^n\dashrightarrow Y\subset \mathbb{P}^N
\]
be a special birational transformation of type $(m,r)$, where $m\ge 2$. If the base locus $S\subset \mathbb{P}^n$ of $\Phi$ is linearly degenerate, then $(m,r)=(2,1)$.
\end{proposition}

\begin{proof}
Let
\[
\pi:W:=\operatorname{Bl}_S(\mathbb{P}^n)\to \mathbb{P}^n,\qquad
\phi:W\to Y
\]
be, respectively, the blow-up along $S$ and the morphism resolving $\Phi$, and let $E$ be the exceptional divisor of $\pi$. Set $H:=\pi^*\mathcal{O}_{\mathbb{P}^n}(1),\ H_Y:=\phi^*\mathcal{O}_Y(1)$. Since $\Phi$ is of type $(m,r)$, we have
\begin{equation}\label{eq:sbt-basic}
H_Y\sim mH-E,\qquad H\sim rH_Y-E_Y,
\end{equation}
where $E_Y$ is the corresponding effective $\phi$-exceptional divisor. Therefore
\begin{equation}\label{eq:EY-class}
E_Y\sim rH_Y-H\sim (mr-1)H-rE.
\end{equation}

By a standard consequence of the negativity lemma, distinct $\phi$-exceptional prime divisors are linearly independent in
$N^1(W/Y)_{\mathbb R}$. Since $\rho(W/Y)=\rho(W)-\rho(Y)=1$, the morphism $\phi$ has at most one exceptional prime divisor. On the other hand, $\rho(W)\ne\rho(Y)$, so $\phi$ is not an isomorphism and therefore has an exceptional prime divisor. Denote this unique exceptional prime divisor by $F$. By \eqref{eq:EY-class}, we have $E_Y\ne0$; hence there exists an integer $b>0$ such that
\begin{equation}\label{eq:EY-bF}
E_Y=bF.
\end{equation}

Since $S$ is linearly degenerate, we can choose a hyperplane
$\Pi\subset\mathbb P^n$ containing $S$. Let $\widetilde{\Pi}\subset W$ be its strict transform. Then
\begin{equation}\label{eq:Pi-class}
\widetilde{\Pi}\sim H-E.
\end{equation}

We prove that $\widetilde{\Pi}=F$. Let $C\subset F$ be a curve contracted by $\phi$.
By \eqref{eq:sbt-basic},
\[
0=H_Y\cdot C=mH\cdot C-E\cdot C,
\]
Thus $E\cdot C=mH\cdot C$. If $H\cdot C=0$, then $C$ is contracted by
$\pi$; but $-E$ is $\pi$-ample, so $E\cdot C<0$, contradicting the equality above.
Therefore $H\cdot C>0$. By \eqref{eq:Pi-class},
\[
\widetilde{\Pi}\cdot C=(H-E)\cdot C=(1-m)H\cdot C<0.
\]
Since $W$ is smooth and $\widetilde{\Pi}$ is an effective divisor, it follows that
$C\subset\widetilde{\Pi}$. Moreover, $F$ is covered by curves contracted by $\phi$, hence
$F\subset\widetilde{\Pi}$. Since both are prime divisors, we obtain $F=\widetilde{\Pi}$.

Thus, by \eqref{eq:EY-bF} and \eqref{eq:Pi-class}, $E_Y=bF=b\widetilde{\Pi}\sim b(H-E)$. Comparing this with \eqref{eq:EY-class}, we obtain
\[
mr-1=b=r.
\]
Thus $r(m-1)=1$. Since $m$ and $r$ are positive integers and $m\ge2$, we must have $r=1,\ m=2$. Therefore $\Phi$ is of type $(2,1)$.
\end{proof}

Fu--Hwang gave a complete classification of special birational transformations of type $(2,1)$:
\begin{theorem}[{\cite[Theorem~1.1]{9}}]\label{26-1-2-683}
Let $\Phi:\mathbb{P}^n\dashrightarrow Y\subset \mathbb{P}^N$ be a special birational transformation of type $(2,1)$. Then its base locus $S^d$ is linearly degenerate, and $S^d\subset \mathbb{P}^{n-1}$ is projectively equivalent to one of the following:

(a) $\mathbb{Q}^d\subset \mathbb{P}^{d+1}$ for $d\ge 1$;

(b) $\mathbb{P}^1\times \mathbb{P}^{d-1}\subset \mathbb{P}^{2d-1}$ for $d\ge 3$;

(c) the 6-dimensional Grassmannian $G(2,5)\subset \mathbb{P}^9$;

(d) the 10-dimensional Spinor variety $\mathbb{S}_5\subset \mathbb{P}^{15}$;

(e) a nonsingular codimension $\le 2$ linear section of $\mathbb{P}^1\times \mathbb{P}^2\subset \mathbb{P}^5$;

(f) a nonsingular codimension $\le 3$ linear section of $G(2,5)\subset \mathbb{P}^9$.
\end{theorem}

\begin{corollary}\label{26-1-2-684}
Let $\Phi:\mathbb{P}^n\dashrightarrow Y\subset \mathbb{P}^N$ be a special birational transformation of type $(m,r)$ with $m\ge 2$. If its base locus $S$ is a linearly degenerate complete intersection, then $S\subset \mathbb{P}^{n-1}$ is projectively isomorphic to a hyperquadric.
\end{corollary}
\begin{proof}
By Proposition~\ref{prop:linearly-degenerate-base-locus}, the transformation must be of type $(2,1)$. The assertion then follows from the classification in Theorem~\ref{26-1-2-683}.
\end{proof}

\begin{corollary}\label{26-1-2-685}
Suppose that for a general $x\in X$, $\mathcal{C}_x$ is a smooth linearly non-degenerate complete intersection, and $I(\mathcal{C}_x)$ is the ideal generated by $F^{m+1}_x$. Then $X\cong \mathbb{Q}^n$.
\end{corollary}
\begin{proof}
By Proposition \ref{26-1-2-681} and Corollary \ref{26-1-2-684}, for a general $x\in X$, $\mathcal{C}_x\subset \mathbb{P}(T_xX)$ is a hyperquadric. Thus, by Theorem \ref{1233}, $X\cong \mathbb{Q}^n$.
\end{proof}

\section{Proof of the Main Theorem}\label{26-1-2-686}

In this section, we synthesize the local differential geometric properties and the classification of special birational transformations established in Sections~\ref{26-1-2-689}--\ref{26-1-2-687} to complete the proof of the Main Theorem.
\begin{proof}
Suppose condition (a) holds. Then by Proposition \ref{26-1-2-677} and Corollary \ref{26-1-2-685}, $X\cong \mathbb{Q}^n$.

Assume that condition (b) holds. By Proposition~\ref{26-1-2-664}, for any minimal rational curve $l$ belonging to $\mathcal{K}$, we have $-K_X\cdot l=i(X)$. Hence
\[
i(X)=\dim(\mathcal{C}_x)+2\ge n-2.
\]
By \cite[Theorem~1.1]{14}, $X$ belongs to the classification list of Fano manifolds of Picard number $1$ that admit an equivariant compactification structure of $\mathbb{C}^n$ and have index at least $n-2$:
\[
\mathbb P^n,\ \mathbb Q^n,\ G(2,5),\ G(2,6),\ \mathbb S_5,\ \mathrm{Lag}(6),
\]
together with smooth linear sections of $G(2,5)$ of codimension $1$ or $2$, and $\mathbb P^4$-general linear sections of $\mathbb S_5$ of codimension $1$, $2$, or $3$.

We first exclude the linear sections of $\mathbb S_5$ of codimension $1$, $2$, or $3$ appearing in the classification table. They can be excluded by the constraint on the dimension of the automorphism group. Indeed, the dimensions of the automorphism groups of these candidates are, respectively, $30$; $17$ or $30$; and $6,7,9$ or $11$ (see
\cite[Introduction]{16};
\cite[Section~3]{17};
\cite[Section~4]{17}). On the other hand, by Proposition~\ref{310}, if these candidates satisfied the assumptions of this paper, then one would have $\dim(\mathrm{Aut}(X))=\dim(X)+1$. In the three codimension cases above, the right-hand side is respectively $10,9,8$, contradicting the known dimensions of the automorphism groups. Therefore these linear sections of $\mathbb S_5$ cannot occur.

We next consider the remaining candidates. First, $\mathbb P^n$ is excluded by the assumption $\mathcal{C}_x\subsetneq \mathbb P(T_xX)$. From the known descriptions of the VMRT at a general point of these homogeneous models and of the linear sections of $G(2,5)$, it follows that $G(2,5)$, $G(2,6)$, $\mathbb S_5$, $\mathrm{Lag}(6)$, and the smooth linear sections of $G(2,5)$ of codimension $1$ or $2$ do not satisfy the assumption required in this paper that the VMRT at a general point is a smooth linearly non-degenerate complete intersection. Therefore the only remaining possibility is $X\cong \mathbb Q^n$.
\end{proof}

\paragraph{Acknowledgements.} I am grateful to Qifeng Li for helpful discussions.

\end{document}